\documentclass[11pt]{amsart}
\usepackage{lscape}
\usepackage{graphicx}
\usepackage{amssymb,amsmath,amsthm,amscd}
\usepackage{mathrsfs}
\usepackage{enumerate}
\usepackage[usenames,dvipsnames]{color}
\usepackage[colorlinks=true, pdfstartview=FitV,  linkcolor=blue,citecolor=blue,urlcolor=blue, bookmarks=false]{hyperref}
\usepackage[all]{xy}

\allowdisplaybreaks[4] 

\numberwithin{equation}{section}

\newcommand{\bj}{\begin{jaune}}
\newcommand{\ej}{\end{jaune}}

\theoremstyle{plain}
\newtheorem{Theorem}{Theorem}[section]

\newtheorem{Lemma}[Theorem]{Lemma}
\newtheorem{Corollary}[Theorem]{Corollary}
\newtheorem*{Lemma*}{Lemma}
\newtheorem*{Theorem*}{Theorem}
\theoremstyle{definition}
\newtheorem{Definition}[Theorem]{Definition}

\newtheorem{Remark}[Theorem]{Remark}
\newtheorem{Remark*}{Remark}
\numberwithin{equation}{section}
\numberwithin{figure}{section}
\numberwithin{table}{section}
\newcommand{\la}{\langle}
\newcommand{\ra}{\rangle}
\newcommand{\lam}{\lambda}
\newcommand{\ZZ}{\mathbb{Z}}
\newcommand{\QQ}{\mathbb{Q}}

\newcommand{\NN}{\mathbb{N}}
\newcommand{\PP}{\mathbb{P}}
\newcommand{\FF}{\mathbb{F}}
\newcommand{\CC}{\mathbb{C}}
\newcommand{\II}{\mathbb{I}}

\newcommand{\bff}{\mathbf{f}}
\newcommand{\bfi}{\mathbf{i}}
\newcommand{\bfj}{\mathbf{j}}

\newcommand{\bfinm}{\mathbf{i}(\nu+\mu)}
\newcommand{\bfk}{\mathbf{k}}

\newcommand{\cP}{\mathcal{P}}
\newcommand{\cQ}{\mathcal{Q}}

\newcommand{\cA}{\mathcal{A}}
\newcommand{\cO}{\mathcal{O}}

\newcommand{\Dim}{\operatorname{Dim}}
\newcommand{\Hom}{\operatorname{Hom}}
\newcommand{\rdim}{\operatorname{dim}}
\newcommand{\HH}{\mathrm{H}}

\newcommand{\Ch}{\operatorname{Ch}}
\newcommand{\E}{\operatorname{E}}
\newcommand{\K}{\operatorname{K}}
\newcommand{\KP}{\operatorname{KP}}
\newcommand{\Qp}{\operatorname{Q}^+}
\newcommand{\Rep}{\operatorname{\mathbf{Rep}}}
\newcommand{\Ext}{\operatorname{Ext}}
\newcommand{\GL}{\operatorname{GL}}
\newcommand{\Gr}{\operatorname{Gr}}
\newcommand{\IC}{\operatorname{IC}}
\newcommand{\odeg}{\operatorname{deg}}
\newcommand{\oht}{\operatorname{ht}}

\newcommand{\ovQ}{\overrightarrow{Q}}
\newcommand{\undd}{\mathrm{\underline{dim}}}

\newcommand{\npm}{{\nu+\mu}}
\newcommand{\ndm}{{\nu,\mu}}

\newcommand{\defeq}{\overset{\operatorname{\scriptstyle def.}}{=}}
\newcommand{\Uq}[1][{\mathfrak{g}}]{{U_q(#1)}}
\newcommand{\Uqm}[1][{\mathfrak{g}}]{{U_q^-(#1)}}
\newcommand{\nq}[1][{q}]{{#1}^{-1}}
\newcommand{\fA}{\textbf{f}_{\mathcal{A}}}

\title{On the cohomology of quiver Grassmannians for acyclic quivers }
\author{Yingjin Bi}

\address{
	School of Mathematical Sciences\\ Beijing Normal University\\ 
	Beijing 100875 P.R.China}
\email{yingjinbi@mail.bnu.edu.cn}
\urladdr{\textrm{\textit{ORCiD}:} \href{https://orcid.org/0000-0003-0153-3274}{orcid.org/0000-0003-0153-3274}}
\keywords{quiver Grassmannian, dual canonical basis, Poincar\'e polynomial,
quantum group} 

\subjclass[2010] {16G20, 14M15, 17B37}

\thanks{\textbf{Declarations}: The author don't recieve any fund during this work; Any data in this paper is available; This paper is a original paper completed independently.}

\begin{document}

\begin{abstract}
For an acyclic quiver, we establish a connection between the cohomology of quiver Grassmannians and the dual canonical bases of the algebra $\Uqm$, where $\Uqm$ is the negative half of the quantized enveloping algebra associated with the quiver. In order to achieve this goal, we study the cohomology of quiver Grassmannians by Lusztig’s category. As a consequence, we describe explicitly the Poincar\'e polynomials of rigid quiver Grassmannians in terms of the coefficients of dual canonical bases, which are viewed as elements of quantum shuffle algebras. By this result, we give another proof of the odd cohomology vanishing theorem for quiver Grassmanians. Meanwhile, for Dynkin quivers, we show that the Poincar\'e polynomials of rigid quiver Grassmannians are the coefficients of dual PBW bases of the algebra $\Uqm$.       
\end{abstract}
\date{}
\maketitle

\section*{Introduction}\label{Intr}
The purpose of this paper is to connect the cohomology of the quiver Grassmannians with the Luszitg's category, which is the geometrization of the half part of the quantized enveloping algebras. 

The cohomology of quiver Grassmannians plays an important role in cluster algebras. In~\cite{CC} and~\cite{Pla}, Caldero,  Chapoton and Plamondon show that the cluster variables and cluster monomials can be expressed in terms of the Euler characteristics of quiver Grassmannians, which is called CC-formula. For the quantum cluster algebras, Qin proves that the quantum cluster monomials can be described by the Serre polynomials of rigid quiver Grassmannians, which is called Quantum CC-formula~\cite{Qin}. Meanwhile, he also proved the odd cohomology vanishing of rigid quiver Grassmannians (Nakajima also showed this fact in~\cite{Nak}). Basing on these works and Nakijima's works~\cite{Nak}, Kimura and Qin prove that the quantum cluster monomials are contained in the dual canonical basis of quantum unipotent subgroups in some case of acyclic quivers~\cite{KQ}. 

Recently, for an acyclic quiver, Irelli, Esposito, Franzen and Reineke show that the cohomology ring of a rigid quiver Grassmannian has property (S), i.e: there is no odd cohomology and the cyclic map is an isomorphism~\cite{IEFR}. However, the property (S) of rigid quiver Grassmannians is not enough to describe the cohomology of rigid quiver Grassmannians explicitly. \\

\noindent
In this paper, we will show that the Poincar\'e polynomial of quiver Grassmannians can be expressed by the coefficients of the dual canonical bases of the algebra $\Uqm$ in acyclic case. It is known that the Poincar\'e polynomial of a quiver Grassmannian can be expressed in the terms of a coefficients of a dual canonical base by Quantum CC-formula~\cite{KQ}. However, for a given quantum cluster monomial, they didn't exactly find out the corresponding dual canonical base. Thus our results will reveal more information on the cohomology of rigid quiver Grassmannians than previous results. \\

\noindent
Let us to illustrate the idea: we fix a field to be the complex number field $\mathbb{C}$. Let $Q=(I, \Omega)$ be an acyclic quiver, where $I$ is the set of vertices and $\Omega$ is the set of arrows. Denote the source (resp:target) of an arrow $h$ by $s(h)$ (resp: $t(h)$). Given an element $\nu=\sum_i\nu_i i\in \NN[I]$, we define a representation space with the element $\nu$ by
\[
\E_{\nu}(Q)\defeq \mathop{\bigoplus}\limits_{ h\in \Omega}\Hom(V_{s(h)},V_{t(h)})
\]
where $V$ is an $I$-graded vector space such that $\rdim V_i=\nu_i$ for all $i\in I$, we denotes $\undd V=\nu$ as its dimension.

For two elements $\mu, \nu\in \NN[I]$, we define a variety as
 \begin{center}
 $\E_{\nu,\mu}\defeq \{(W,y)\mid \text{ $y\in \E_{\nu+\mu}(Q)$, $y_h(W_{s(h)})\subseteq W_{t(h)}$ and  $\undd W=\mu$ }\}$
 \end{center}
We construct a map as follows:
\begin{equation}\label{eq 1}
\begin{split}
 p_\mu: \  \E_{\nu,\mu} &\to \E_{\nu+\mu}(Q)\\
(W,y)&\mapsto y 
\end{split}
\end{equation}
Given a point $M\in \E_{\nu+\mu}(Q)$,  we define the \emph{quiver Grassmannians} as the fiber of the map $p_\mu$ at $M$,e.t. $p_\mu^{-1}(M)$, denoted as $\Gr_\mu(M)$. Since $p_\mu$ is proper, we have $\Gr_\mu(M)$ is a projective variety. We assume its cohomology $\HH^\bullet(\Gr_\mu(M))$ is defined over the complex number field $\CC$. 

Next we consider the Luszitg's category on the variety $\E_{\nu}(Q)$. For an element $\nu\in \NN[I]$, Lusztig introduced a set of simple perverse sheaves, denoted $\cP^\nu$, on $\E_\nu(Q)$. By Lusztig's geometric construction of canonical bases of the algebra $\Uqm$, $\cP^\nu$ corresponds to its canonical basis. One denotes the canonical base by $b_\lam$ for each element $\lam\in \cP^\nu$. 

Since the variety $\E_{\nu}(Q)$ is smooth,  there is a constant perverse sheaf
$$\II_\nu\defeq \CC_{\E_\nu(Q)}[\rdim \E_{\nu}(Q)]$$ e.t., that is the constant sheaf $\CC_{\E_\nu(Q)}$ shifted up $\rdim \E_{\nu}(Q)$. Moreover, Since the quiver $Q$ is an acyclic quiver, it is in the set $\cP^\nu$. (see~\cite{Schi}). We also have that there is a constant perverse sheaf $\mathbb{C}_{\E_{\nu,\mu}}[\rdim \E_{\nu,\mu}]$ on the variety $\E_{\nu,\mu}$ for it is a smooth variety. One denotes the constant perverse sheaf as $\mathbb{I}_{\E_{\nu,\mu}}$.\\

The following the results is a key step for our results. 
\begin{Lemma*}\cite[Lemma 8.5.4]{CG}
 Under above assumption, for a representation $M$ with dimension ${\mu+\nu}$, we have
\[
\begin{split}
\HH^\bullet(\Gr_\mu(M))\cong \HH^{\bullet-\rdim \E_{\nu,\mu}}(i_M^*p_{\mu !}\mathbb{I}_{\E_{\nu,\mu}})
\end{split}
\]
where $i_M: \{M\}\hookrightarrow \E_{\nu+\mu}$ is the embedding corresponding to the point $M$.
\end{Lemma*}
In~\cite{Schi}, the complex $p_{\mu !}\mathbb{I}_{\E_{\nu,\mu}}$ can be viewed as the multiplication of Lusztig's sheaves. \\

\noindent
In~\cite{Lec}, Leclerc shows that the positive part $\Uq[{\mathfrak{n}}]$ of a quantum enveloping algebra $\Uq$ can be embedded in a quantum shuffle algebra in finite type cases. Thus the dual canonical bases of the algebra $\Uq[{\mathfrak{n}}]$ could be viewed as elements of quantum shuffle algebras. However, For our purpose and our assumption, It's better to consider the dual canonical bases as the character of the simple modules over the corresponding KLR algebra~\cite{VV}, because the cohomology of the complex $\HH^{\bullet}(i_M^*p_{\mu !}\mathbb{I}_{\E_{\nu,\mu}})$ is always called a dual standard module in geometric representation theory. (see examples~\cite{Kato},~\cite{Nak}).

One sets a total order of an index of $I$ so that $k>l$ if and only if there exists an arrow from $i_l$ to $i_k$ (it is possible for acyclic quivers).
For an element $\nu\in\NN[I]$, we define a set $\la I\ra_\nu\defeq \{[j_1j_2\cdots j_l]| \sum_{0\leq k\leq l}j_k=\nu\}$ where the $[j_1j_2\cdots j_l]$ are words on $I$. For an element $\lambda\in \cP^\nu$, we can express a dual canonical base $b_\lam^*$ as $$b_\lambda^*=\sum_{\bfk\in \la I\ra_{\nu}}\chi_\lambda^\bfk \bfk$$ where $\chi_\lambda^\bfk\in \ZZ[q,q^{-1}]$ satisfy $\chi_\lambda^\bfk(q)=\chi_\lambda^\bfk(\nq)$.

For our purpose, we always consider the element of $\NN[I]$ as form $\nu+\mu$ for some $\ndm\in\NN[I]$. For an element $\npm=\sum_{i_k\in I}(\npm)_{k}i_k$, there exists a canonical base $b=f_{i_1}^{((\npm)_1)}\cdots f_{i_n}^{((\npm)_n)}$ where $f_{i_k}$ are generators of the algebra $\Uqm$ associated with $Q$, we denote $\bfi(\npm)$ as the maximal word of its dual canonical base $b^*$(this word is unique in the set $\la I\ra_{\npm}$). It is convenient to rename the dual canonical base $b^*$ as $b_{\bfi{(\npm)}}^*$. Next we will give another special words in $\la I\ra_\nu$. For an element $\nu=\sum_{i_k\in I}\nu_k i_k\in\NN[I]$, we set a word
 \[
\bfi_\nu=[i_1\cdots i_1i_2\cdots i_n\cdots i_n]
\]  
where the multiplicity of $i_k$ in the word is equal to $\nu_k$ for each $k\in I$. For two words $\bfk=[k_1k_2\cdots k_m]$ and $\bfj=[j_1j_2\cdots j_l]$, we define air concatenation by $\bfk\bfj=[k_1k_2\cdots k_m j_1j_2\cdots j_l]$. Thus we have that $\bfi_\nu\bfi_\mu\in \la I\ra_{\npm}$. For more details, we can refer to the notations in Lemma~\ref{lemma schi} and Remark~\ref{dual canonical bases}.

 Fix an indeterminate $q$. For a representaion $M$ of the quiver $Q$ with dimension $\undd{(M)}=\npm$, we define the Poincar\'e polynomial of the quiver Grassmannian $\Gr_\mu(M)$ as \[\Ch_q(\Gr_\mu(M))\defeq\sum_{i\in\mathbb{Z}}\rdim \HH^i(\Gr_\mu(M))q^i\] 

\begin{Theorem*}
Let $M$ be a rigid representation of an acyclic quiver $Q$ with dimension $\undd(M)=\npm$. then we have 
\begin{align}
  \Ch_q(\Gr_\mu(M))=q^{s_\nu+s_\mu-t_{\ndm}-\la\mu,\nu\ra}([\nu]![\mu]!)^{-1}\chi_{\bfi(\npm)}^{\bfi_\nu\bfi_\mu}
\end{align}
where $t_{\ndm}$ is equal to $\rdim \E_\nu+\rdim \E_\mu$, $\chi_{\bfi(\npm)}^{\bfi_\nu\bfi_\mu}$ is the coefficient of the word $\bfi_\nu\bfi_\mu$ in the expression of the dual canonical base $b_{\bfi(\npm)}^*$, and $\la\mu,\nu\ra$ refers to the Euler form of $Q$. Here $s_\nu+s_\mu$ and $[\nu]![\mu]!$ will be given in~\ref{setting of vector}. 
\end{Theorem*}
For Dynkin quivers, we calculate the Poincare polynomial of rigid quiver Grassmannians explicitly as follows.
\begin{Theorem*}
Let $M$ be a rigid representation of a Dynkin quiver $Q$ with dimension $\undd(M)=\npm$. then we get 
\begin{align}
  \Ch_q(\Gr_\mu(M))=q^{s_\nu+s_\mu-t_{\ndm}-\la\mu,\nu\ra}/[\nu]![\mu]!h_{(\npm)^0}^{\bfi_\nu\bfi_\mu}
 \end{align} 
 where $h_{(\npm)^0}^{\bfi_\nu\bfi_\mu}$ is the coefficient of the word $\bfi_\nu\bfi_\mu$ in the dual PBW base of $r_{(\npm)^0}^*=b_{\bfi(\npm)}^*$ as above. 
\end{Theorem*}
The notations in above theorem are given in section~\S\ref{Dynkin}. The reason of replacing the dual canonical base by a dual PBW base is to make it easy to calculate the coefficients $\chi_{\bfi(\npm)}^{\bfi_\nu\bfi_\mu}$ in the Dynkin case.\\

\noindent
\emph{Organization} 
This paper is organized as follows.
In the section~\ref{first}, we review canonical basis and dual canonical basis of the algebra $\Uqm$ associated with an acyclic quiver.
In the secion~\ref{acyclic}, we prove the main theorem of this paper. In the section~\ref{antherproof}, we give a new proof of the odd cohomology vanishing theorem for quiver Grassmannians.
In the section~\ref{Dynkin},  we compute the Poincare polynomial of rigid quiver Grassmannians by the dual PBW bases for Dynkin quivers. At last, we focus on the type $A$ and calculate this polynomial explicitly. \\

\noindent
{\bf Acknowledgements.} 
The author sincerely thanks professor Giovanni Cerulli Irelli and my advisor Yuming Liu for providing a lot of help during this work. Giovanni Cerulli Irelli firstly sees the primary draft and gives some quite valuable comments on this paper. He would also like to express his gratitude to professor Fan Qin for offering some valuable suggestions. Meanwhile, he would like to thank professor Bernard Leclerc, for his comments provide some clues for this paper.

\section{Canonical bases and dual canonical bases on acyclic quivers}\label{first}

In this section, we will recall some basic results on acyclic quivers. Let $\overrightarrow{Q}=(I,\Omega)$ be an acyclic quiver, where $I$ is the set of the vertices of the quiver and $\Omega$ is the set of the arrows of the quiver. A quiver is called acyclic if and only if it has no cycles. We denote by $Q$ the underlying graph of the quiver $\overrightarrow{Q}$. For $h\in \Omega$ we write $s(h)$ and $t(h)$ for the source of $h$ and the target of $h$, respectively. We fix an orientation $\Omega$ of an acyclic quiver $\ovQ$ and set a total order on the set of the vertices $I$ so that $i_k>i_l$ if there exists an arrow $h:i_l\to i_k$, then we can index the set of the vertices so that $i_k>i_l$ if and only if $k>l$. We denote $n$ as the  the number of $I$ in this section.\\

 One can define a representation of the quiver $\overrightarrow{Q}$ over a field $k$ in the following way: it is a tuple $(M,(M_h)_{h\in \Omega})$ where $M$ is an $I$-graded vector space over $k$ and $M_h:M_{s(h)}\to M_{t(h)}$ is a linear map. We denote by $\Rep_k(\overrightarrow{Q})$ the category of representations of the quiver. Given two representations $M,N\in \Rep_k(\overrightarrow{Q})$, we denote by $\Hom_{\overrightarrow{Q}}(M,N)$ the vector space of $\overrightarrow{Q}$-morphisms between $M$ and $N$ and write $[M,N]$ (resp: $[M,N]^1$) for $\rdim(\Hom_{\overrightarrow{Q}}(M,N))$ (resp: $\rdim(\Ext^1_{\overrightarrow{Q}}(M,N))$).

For a representation $M\in \Rep_k(\overrightarrow{Q})$, we write $\undd M$ for its dimension $\sum_{i_k\in I}\rdim(M_{i_k}) i_k$. Given two representations $M,N$ of the quiver, we define an integral bilinear form $\langle M,N\rangle$ as
\begin{equation}\label{2.1}
 \langle M,N\rangle =\langle\undd M,\undd N\rangle\defeq\sum_{i_k\in I}\rdim M_{i_k} \rdim N_{i_k}-\sum_{h\in \Omega}\rdim M_{s(h)} \rdim N_{t(h)}  
\end{equation}

For an acyclic and connected quiver (such as a Dynkin quiver), one obtains the following equation:
\begin{align}\label{2.2}
  \langle M,N\rangle=[M,N]-[M,N]^1
\end{align}

Based on the above notations, we define a Cartan matrix $A_Q=(a_{i,j})_{i,j\in I}$ for the quiver $Q$ (which only depends on the underlying graph) by
\begin{align}\label{Cartan matrix}
  a_{i,j}=\begin{cases}
   2&\text{$i=j$}\\
  \langle S(i), S(j)\rangle +\langle S(j), S(i)\rangle&\text{$i\neq j$}
  \end{cases}
\end{align}
where $S(k)$ denotes the simple module corresponding to $k\in I$ (Here we ignore the order of the vertex set $I$).

One also defines the Kac-Moody algebra $\mathfrak{g}_{Q}$ associated with the Cartan matrix $A_Q$, which generated by a vector space $\mathfrak{b}$, elements $\{e_i\}_{i\in I}$ and $\{f_i\}_{i\in I}$ satisfy some conditions. We write $R^+$ for the set of its positive roots. Meanwhile, for simple roots $\alpha_i$ and $\alpha_j$, their inner product is given by $\alpha_i\cdot\alpha_j=a_{i,j}$ 

For a Cartan matrix $A_Q$ and an indeterminate $q$, we write $\textbf{f}_{\mathcal{A}}$ for Lusztig's integral form of the negative half of the quantum universal enveloping algebra associated with the quiver $Q$, where $\mathcal{A}=\mathbb{Z}[q,q^{-1}]$. Given a number $m\in \mathbb{N}$, we set

\begin{align}
[m]\defeq \frac{q^m-q^{-m}}{q-q^{-1}} &\qquad  [m]!\defeq \prod_{1\leq k\leq m} [k]
\end{align}

The algebra $\textbf{f}_{\mathcal{A}}$ is a $\Qp\defeq \sum_{i\in I}\NN{\alpha_i}$-graded algebra and is generated by elements $\{\theta_i\}_{i\in I}$ subject to the quantum Serre relations

\begin{align}
\mathop{\sum}\limits_{r+s=1-a_{ij}}(-1)^r\theta_i^{(r)}\theta_j\theta_i^{(s)}=0
\end{align}
 for all $i,j\in I$ and $r\geq 1$, where $\theta_i^{(r)}$ denotes the \emph{divided power} $\theta_i^{r}/[r]!$. Moreover, we write $\bff$ for the algebra freely genereted by $\theta_{i}$ for all $i\in I$. There exists a nondegenerate symmetric bilinear form $(\cdot,\cdot)$ on $\textbf{f}_{\mathcal{A}}$  such that $(1,1)=1$ and

 \begin{equation}\label{def:bilinearform}
 (e_i'(u),v)=(u,\theta_i v),\qquad (u,v\in \fA,i=1,\cdots,r)
 \end{equation}
 where $e_i'$ is the $q$-derivations of $\fA$ given by
 \begin{equation}
 e_i'(\theta_j)=\delta_{ij},\qquad e_i^i(uv)=e_i'(u)v+q^{(\alpha,|u|)}ue_i'(v),
 \end{equation}
for all homogeneous elements $u,v$ of $\fA$, where $|u|$ is the $\Qp$-degree.

\subsection{Lusztig's geometric construction of canonical bases}\label{Lusztig's construction}
 The original results concerning geometric construction are given in the cases of the closure field of a finite field. However, the same results are obtained over the complex number field by~\cite[Chapter 6]{BBD} via Hall category (see~\cite[Remark 3.27]{Schi}). Thus unless specified otherwise, we assume the ground field is the complex number field $\mathbb{C}$ and replace the field $\overline{\mathbb{Q}}_l$ with $\mathbb{C}$. Due to the results of this section just depend on the underlying graph $Q$, we abbreviate the $\overrightarrow{Q}$ as $Q$.\\

\noindent
For an element $\nu=\sum_{i_k\in I}\nu_k i_k \in \NN[I]$, Set an affine space $\E_\nu(Q)\defeq \bigoplus_{ h\in \Omega}\Hom(V_{s(h)},V_{t(h)})$, where the vector spaces $V_i$ satisfy $\rdim V_i=\nu_i$ for $i\in I$, and the linear algebraic group $\GL(\nu)\defeq \Pi_{i\in I}\GL(V_i)$ acts on it. Lusztig constructed a family of $\GL(\nu)$-equivariant perverse sheaves, denoted by $\mathcal{P}^\nu$, on $\E_\nu$ to describe the canonical bases of the algebra $\mathbf{f}_\cA$. For acyclic quivers, the elements of $\cP^\nu$ are of the form $\IC(Y,\lam)$, where $Y$ is a smooth locally closed $\GL{(\nu)}$-invariant subvariety and $\lam$ is an irreducible local system on $Y$. Explicitly, we have $\IC(Y,\lam)=i_*j_{!*}\lam[\rdim Y]$ where $j_{!*}$ stands for the intermediate Extension of the open embedding $j:Y\hookrightarrow \overline{Y}$ and here $i:\overline{Y}\hookrightarrow \E_\nu$ (see~\cite[Definition 5.2]{KW}). From now on, we write $\lam$ or $\IC(\lam)$ for an element $\IC(Y,\lam)\in \cP^\nu$. We abbreviate space $\E_\nu(Q)$ as $\E_\nu$ if there is no danger of confusion 
\begin{Lemma}\label{exam3.1}
Given an element $\nu \in \NN[I]$, the perverse sheaf $\mathbb{C}_{\E_\nu}[\rdim \E_\nu]$ on $\E_\nu$ is contained in $\mathcal{P}^\nu$ (\emph{see~\cite[Example 2.5]{Schi}}). If there exists a unique open $\GL(\nu)$--orbit $\mathcal{O}_{max}$ in $\E_\nu$, we have
\begin{center}
$\mathbb{C}_{\E_\nu}[\rdim \E_\nu]=\IC(\mathcal{O}_{max},\CC)$
\end{center}
From now on, we denote $\mathbb{C}_{\E_\nu}[dim \E_\nu]$ as  $\mathbb{I}_\nu$.
\end{Lemma} 
\begin{proof}
Since $\mathbb{C}_{\E_\nu}[\rdim \E_\nu]$ is simple,  by~\cite[Corollary 5.4]{KW} there is a simple perverse sheaf $A$ on the open subset $\cO_{max}$ such that $$\mathbb{C}_{\E_\nu}[\rdim \E_\nu]=j_{!*}A$$
 where $j$ is the open embedding $j:\mathcal{O}_{max}\hookrightarrow \E_\nu$. 

Since the stablizer of $x\in \cO_{max}$ is connected, it follows that the unique irreducible local system of $\cO_{max}$ is the field $\CC$, thus we have $A=\mathbb{C}_{\cO_{max}}[\rdim \cO_{max}]$. Hence we obtain $\mathbb{C}_{\E_\nu}[\rdim \E_\nu]=\IC(\mathcal{O}_{max},\CC)$..
\end{proof}

Lusztig also defined the multiplication of the perverse sheaves, one should see~\cite{Lubook}for the details. The simple vision is given as follows: Given two elements $\nu,\mu \in \NN[I]$, we define a variety
 \begin{align}\label{3.1}
 \E_{\nu,\mu}\defeq\{(W,y)\mid \  \text{ $y\in \E_{\nu+\mu}$; $y_h(W_{s(h)})\subseteq W_{t(h)}$ and  $\undd W=\mu$ }\}
 \end{align}
 where $W$ are $I$-graded vector spaces. It is a smooth variety with the constant perverse sheaf $\II_{\E_\ndm}=\CC_{\E_\ndm}[\rdim \E_\ndm]$. We construct a map:
\begin{equation}\label{main map}
\begin{split}
p: & \E_{\nu,\mu} \to \E_{\nu+\mu}\\
& (W,y)\mapsto y
\end{split}
\end{equation}

The map $p$ is a proper morphism. The multiplication of $\mathbb{I}_\nu$ and $\mathbb{I}_\mu$ is given by
\begin{align}\label{3.2}
\mathbb{I}_\nu\star \mathbb{I}_\mu=p_{!}\mathbb{C}_{\E_{\nu,\mu}}[\rdim \E_{\nu,\mu}+\rdim \E_\nu+\rdim \E_\mu]
\end{align}

This definition of the multiplication coincides with the original Lusztig's definition (see~\cite[Section 1.4]{Schi}).

We denote by $\mathcal{Q}_\nu$ the semisimple subcategory of the derived category $D_{\GL(\nu)}^b(\E_\nu)$ generated by $\mathcal{P}^\nu$. There is a $\ZZ[q,q^{-1}]$-structure on the Grothendieck group of the category $\cQ_\nu$, which is denoted by $K_q(\cQ_\nu)$, given by
\begin{equation}\label{liftq}
	q^m[\PP]=[\PP[m]]; \qquad\text{for any $[\PP]\in K_0(\cQ_\nu)$}  
\end{equation}

Lusztig's geometric construction of canonical bases is given as follows.
\begin{Theorem}\cite[Theorem 3.6]{Schi}\label{theom3.2}
Under above assumption, there exists an algebraic isomorphism $\bigoplus_{\nu\in \NN[I]}K_{q^{-1}}(\mathcal{Q}_\nu)\cong \mathbf{f}_{\mathcal{A}}$ under the above multiplication. The elements in $\mathbf{f}_{\mathcal{A}}$ corresponding to the elements of $\mathcal{P}^\nu$ is called canonical bases.
\end{Theorem}

Note that category $\mathcal{Q}_\nu$ doesn't depend on the ground field $k$. By the facts in~\cite{BBD}, we know this category defined over $\mathbb{C}$ is equivalent to this category defined over algebraically closed fields $\overline{\mathbb{F}_q}$, or see~\cite[Remark 3.27]{Schi}. Meanwhile, we should pay attention to the parameter $q^{-1}$ in the above Grothendieck group $K_{q^{-1}}(\mathcal{Q}_\nu)$, which is the inverse of $q$ of the algebra $\mathbf{f}_{\mathcal{A}}$.

From now on, we write $b_\lambda$ for the canonical base in $\textbf{f}_{\mathcal{A}}$ corresponding to the perverse sheaf $\IC(\lambda)$. In particular, we write $b_{\nu^0}$ for the canonical base in $\textbf{f}_{\mathcal{A}}$ corresponding to the perverse sheaf $\mathbb{I}_\nu$.
As a corollary, we have
\begin{align}\label{3.5}
b_{\nu^0}\star b_{\mu^0}=\sum_{\lambda\in \cP^{\npm}}\chi_{\nu,\mu}^{\lambda}b_\lambda
\end{align}
where $\chi_{\nu,\mu}^{\lambda}\in \mathbb{N}[q,q^{-1}]$. There is another way to describe the coefficients $\chi_{\nu,\mu}^{\lambda}\in \mathbb{N}[q,q^{-1}]$ by Hall algebra of the category $\mathcal{Q}_{Q}=\bigoplus_{\nu\in \NN[I]}\mathcal{Q}_\nu$ (see~\cite{Schi}).

\subsection{Transformation of the product of constant perverse sheaves}\label{Transformation of product}
In this section, we will make some reduction for Formula~\ref{3.2}.
\begin{Lemma}\cite[Example 2.5]{Schi}\label{lemma schi}
For each element $\nu=\sum_k \nu_k i_k$, the perverse sheaf $\mathbb{I}_\nu=\mathbb{C}_{\E_\nu}[\rdim \E_\nu]$ is of the form
\[
L_{\nu_1 i_1,\cdots ,\nu_n i_n}=\mathbb{I}_\nu
\]
where $L_{\nu_1i_1,\cdots ,\nu_n i_n}$ is introduced in~\cite{Lubook}.  It is known that the canonical base of the algebra $\bff_\cA$ corresponding to the perserve $L_{\nu_1i_1,\cdots ,\nu_n i_n}$ is equal to $\theta_\nu=\theta_{i_1}^{(\nu_1)}\cdots \theta_{i_n}^{(\nu_n)}$. \emph{Note the ordering of $I$ is given at the beginning of~\S\ref{first}, $n=|I|$ and $b_{\nu^0}=\theta_\nu$}
\end{Lemma}
For our purpose, we introduce the concept of words on $I$. We set $\langle I \rangle$ be the free monoid on $I$, e.t., the set of all words $\bfi=[j_1\cdots j_l]$ for $l\geq 0$ and $j_1,\cdots, j_l\in I$ with multiplication given by concatenation of words: for two words $\bfj=[j_1\cdots j_m]$ and $\bfk=[k_1\cdots k_l]$, its concatenation is defined by 
\begin{align}\label{concatenation of word}
\bfj\bfk=[j_1\cdots j_m k_1\cdots k_l]
\end{align}

For an element $\nu=\sum_{i_k}\nu_k i_k\in \NN[I]$, we set
\begin{equation}\label{setting of vector}
\begin{split}
S_\nu\defeq  \prod_{1\leq k\leq n} S_{\nu_k}; \qquad & \quad s_\nu\defeq  \sum_{1\leq k\leq n}
 \frac{1}{2}(\nu_k-1)\nu_k;\\
 [\nu]!\defeq  \prod_{1\leq k\leq n} [\nu_k]!; \qquad  &\quad \bfi_\nu=[i_1\cdots i_1i_2\cdots i_n\cdots i_n]
 \end{split}
\end{equation}
where $S_{\nu_i}$ stands for the symmetric group of $\nu_i$ and the multiplicity of $i_k$ in $\bfi_\nu$ is $\nu_k$. In~\cite{Lubook}, for any words $\bfi\in \la I\ra_\nu$ there is a corresponding Lusztig's sheaf $L_\bfi\in \cQ_\nu$ satisfying $L_\bfi\star L_\bfj=L_{\bfi\bfj}$ 

By~\cite[Remark 1.5]{VV} and Lemma~\ref{lemma schi}, we have an isomorphism of complexes
\begin{align}\label{perverse product 1.2}
L_{\bfi_\nu}\cong\mathop{\bigoplus}\limits_{w\in S_{\nu}}\mathbb{I}_\nu[-2l(w)]
\end{align}
where $l(w)$ stands for the length of the element $w\in S_\nu$. 

Since the Poincar\'e polynomial of the symmetric group $S_n$ is that $\mathop{\sum}\limits_{w\in S_n}q^{2l(w)}=q^{\frac{1}{2}n(n-1)}[n]!$, it follows
\begin{align}\label{element setting 1}
\mathop{\sum}\limits_{w\in S_\nu}q^{2l(w)}=q^{s_\nu}[\nu]!
\end{align}

By the above discussion, we have
\begin{equation}\label{reduction 1}
\begin{split}
L_{\bfi_\nu\bfi_\mu}=L_{\bfi_\nu}\star L_{\bfi_\mu}
 =&\left(\bigoplus_{w\in S_\nu} \II_\nu[-2l(w)]\right)\star\left(\bigoplus_{v\in S_\mu} \II_\mu[-2l(v)]\right)\\
 =& \left(\bigoplus_{w\in S_\nu}q^{-2l(w)}\II_\nu\right)\star\left(\bigoplus_{v\in S_\mu}q^{-2l(v)}\II_\mu\right)\\
 =& q^{-s_\nu-s_\mu}[\nu]![\mu]!\II_\nu\star\II_\mu \qquad \text{by the equation~\ref{element setting 1}}
\end{split}
\end{equation}
where $\bfi_\nu\bfi_\mu$ is the concatenation of $\bfi_\nu$ and $\bfi_\mu$.

\subsection{Quantum shuffle algebras}\label{shuffle algebra}
In this section, we briefly reviews basic results on dual canonical basis in terms of words on $I$. It become available to describe the cohomology of quiver Grassmannians in terms of dual canonical bases via words on $I$.

Under the same assumption of above subsection. For a word $\bfj=[j_1\cdots j_m]$ of length $m$ and a permutation $w\in S_m$, we set
\begin{equation}\label{setting shuffle}
\begin{split}
|\bfj|\defeq \alpha_{j_1}+\cdots +\alpha_{j_m}, \qquad &w(\bfj)\defeq  [j_{w^{-1}(1)}\cdots j_{w^{-1}(m)}] \\
\theta_{\bfj}\defeq \theta_{j_1}\cdots\theta_{j_m},\qquad &\odeg(w;\bfj)\defeq  -\mathop{\sum}\limits_{\mathop{1\leq l<k\leq m}\limits_{w(l)>w(k)}}\alpha_{j_l}\cdot\alpha_{j_k}
\end{split}
\end{equation}
where $\alpha_{j_i}$ is the simple root of algebra $\mathfrak{g}_Q$ corresponding to the vertex $j_i$ and $\alpha_{j_l}\cdot\alpha_{j_k}$ is given by $\alpha_{j_l}\cdot\alpha_{j_k}=a_{j_l,j_k}$ (see~\ref{Cartan matrix})

 For an element $\nu\in \NN[I]$, set $\langle I\rangle_\nu\defeq \{\bfi|\ |\bfi|=\nu\}$, and the monomials $\{\theta_{\bfi}| \ \bfi\in \langle I\rangle_\nu\}$ span $\textbf{f}_\nu$. The \emph{quantum shuffle algebra} is a free $\mathcal{A}$-module $\mathcal{A}\langle I\rangle=\bigoplus_{\nu\in \NN[I]}\mathcal{A}\langle I\rangle_\nu$ on basis $\langle I\rangle$, viewed as an $\mathcal{A}$-algebra via the shuffle product $\circ$ defined on words $\bfi$ and $\textbf{j}$ of length $m$ and $l$, respectively, by
\begin{align}\label{shuffle product}
\bfi\circ\textbf{j}\defeq \mathop{\sum}\limits_{\mathop{w\in S_{l+m}}\limits_{\mathop{w(1)<\cdots<w(m)}\limits_ {w(m+1)<\cdots<w(m+l)}}} q^{\odeg(w;\textbf{ij})}w(\textbf{ij})
\end{align}
We define the dual algebra of $\fA$ as $\bff_\cA^*=\Hom_\cA(\bff_\cA,\cA)$. By the nondegenerate symmetric bilinear form $(\cdot,\cdot)$ (\ref{def:bilinearform}), there is a canonical isomorphism $\Phi$ given by
\begin{equation}
  	\begin{split}
  	\Phi: &\fA\to \fA^*\\
           &x\mapsto \phi_x(y):=(x,y)\qquad {\text{for all elements $y\in \fA$}}
  	\end{split}
  \end{equation}  
There is an injective $\mathcal{A}$-algebra homomorphism
\begin{equation}\label{Ch}
\begin{split}
\Ch: \textbf{f}_{\mathcal{A}}^*&\to \mathcal{A}\langle I\rangle\\
\phi_x& \mapsto  \mathop{\sum}\limits_{\bfi\in \langle I\rangle}(\theta_{\bfi},x)\bfi. 
\end{split}
\end{equation}
 Generally, for the dual algebra of $\bff$, we can define a map $\Ch: \textbf{f}^*\to \mathbb{Q}(q)\langle I\rangle$ as well. The map is also injective.

Next we define a total order on $\langle I\rangle$ by the lexicographic order arising from the total order on $I$ (see section~\S\ref{first}), which will be denoted by $<$. For two words $\bfj=\{j_1j_2\ldots j_m\}$ and $\bfi=\{i_1i_2\ldots i_n\}$, we say $\bfj>\bfi$ if there exists $k\in \NN$ such that $j_l=i_l$ for $l<k$ and $j_k>i_k$ or $\bfj=\bfi\bfk$ for some word $\bfk$.
\begin{Definition}
A word $\bfi\in \langle I\rangle$ is called a good word, if there exists an element $x\in \mathbf{f}^*$ such that $\bfi$ is the maximal word in the expression of $\Ch(x)$, e.t., $\bfi=max(\Ch(x))$. Denote by $\langle I\rangle^+$ the set of all good words.
\end{Definition}
\begin{Remark}\label{dual canonical bases}
For each element $\bfi\in\langle I\rangle^+$, there exists a unique dual canonical base $b_\bfi^*$ in $\bff_\cA^*$ such that $max(\Ch(b_\bfi^*))=\bfi$. From then on, we abbreviate $\Ch(b_\bfi^*)$ as $b_\bfi^*$ if no confusion.

Additionally, we denote the maximal word of the dual canonical base $b_{\nu^0}^*$ by $\bfi(\nu)$; that is $b_{\nu^0}^*=b_{\bfi(\nu)}^*=\theta_\nu^*$, (see Lemma~\ref{lemma schi}).
\end{Remark}

\section{the cohomology of quiver Grassmannians}\label{acyclic}
\subsection{The cohomology of quiver Grassmannians for acyclic quivers}\label{acyclic quiver}
 We recall the concept of quiver Grassmannians. In~\S\ref{Lusztig's construction}, for two elements $\ndm\in \NN[I]$, there is a map (see~(\ref{main map}))
\begin{center}
 $p_\mu:\E_{\nu,\mu}\to \E_{\nu+\mu}$
\end{center}
Given a point $M$ in $\E_{\nu+\mu}$, we call the variety $p_\mu^{-1}(M)$ as the quiver Grassmannian with dimension vector $\mu$ for the representation $M$ and denote it by $\Gr_\mu(M)$.

{\bf Notation} For the classical notion, the dimension vector $\mu$ is written as the form
\[\mathbf{e}=(e_{i_1},e_{i_2},\cdots,e_{i_n})\]
where $e_{i_k}\in \NN$ for each $i_k\in I$. However, It is convenient to describe the cohomology of quiver Grassmannians by dual canonical bases when we view a dimension vector as an element of $\NN[I]$.

By~\cite[Lemma 1.4]{Schi}, we have that $\E_{\nu,\mu}$ is a smooth variety with dimension 
\begin{align}\label{dimension formula}
 \rdim \E_{\ndm}=\rdim \E_\npm+\la\mu,\nu\ra
\end{align}
 We denote the perverse sheaf $\mathbb{C}_{\E_{\nu,\mu}}[\rdim \E_{\nu,\mu}]$ as $\mathbb{I}_{\E_{\nu,\mu}}$, the embedding $\{M\}\to \E_{\mu+\nu}$ as $i_M$, and $\rdim \E_\nu+\rdim \E_\mu$ as $t_{\nu,\mu}$. Assume the cohomology of $\Gr_\mu(M)$ is defined over $\CC$ and denote it by $\HH^\bullet(\Gr_\mu(M))$.

\begin{Lemma}\cite[Lemma 8.5.4]{CG}\label{lem5.2}
Let $Q$ be an acyclic quiver, for a representation $M$ of $Q$ with dimension ${\mu+\nu}$, there exists an isomorphism 
\begin{align}\label{cohomology and homology}
\HH^\bullet(\Gr_\mu(M))\cong \HH^{\bullet-\rdim \E_{\nu,\mu}}(i_M^*p_!\mathbb{I}_{\E_{\nu,\mu}})
\end{align}
\end{Lemma}
\begin{proof}
 Follows from~\cite[Lemma 8.5.4]{CG}
\end{proof}
\begin{Lemma}\cite[Decomposition Theorem]{BBD}\label{BBD decomposition}
Under the above assumption, the complex $p_!\mathbb{I}_{\E_{\nu,\mu}}$ is decomposed as a direct sum of simple perverse sheaves (up some shift): $$p_!\mathbb{I}_{\E_{\nu,\mu}}=\bigoplus_{\lambda\in \cP^\npm}V(\lambda)\boxtimes \IC(\lambda)$$ where $V(\lambda)$ are $\mathbb{Z}-$graded vector spaces.
\end{Lemma}
\begin{proof}
Follows from the BBD Decomposition Theorem in~\cite{BBD}.
\end{proof}
To describe the cohomology of quiver Grassmannians, we introduce the concept of the Poincar\'e polynomial of quiver Grassmannians, for it is an accessible way to connect the cohomology of quiver Grassmannians with the dual canonical bases of algebra $\bff_\cA$. 

\begin{Definition}\label{the Poincare polynomial}
The Poincar\'e polynomial of a quiver Grassmannian $\Gr_\mu(M)$ is given by
\[
 \Ch_q(\Gr_\mu(M))\defeq \sum_{n\in \NN} \rdim \HH^n (\Gr_\mu(M)) q^n
\]
\end{Definition} 
Similarly, we can define $\Ch_q(V(\lam))$ and $\Ch_q(i_M^*\IC(\lam))$ as 
\begin{align}
\Ch_q(V(\lam))\defeq  \sum_{n\in\ZZ} \rdim V(\lam)_n q^n &\qquad \Ch_q(i_M^*\IC(\lam))\defeq  \sum_{n\in\ZZ} \rdim \HH^n (i_M^*\IC(\lam))q^n
\end{align}

\begin{Theorem}\label{theorem of computation coefficients}
Let $\nu,\mu\in \NN[I]$ be two elements and $b_\lam$ the canonical base corresponding to an element $\lam\in\cP^\npm$. We express its dual canonical base as $b_\lambda^*=\sum_{\bfk\in \la I\ra_{\npm}}\chi_\lambda^\bfk \bfk$. For the $\ZZ$-graded vector space $V(\lambda)$ in Lemma~\ref{BBD decomposition}, we have 
\begin{align}\label{computation coefficient vector spaces}
 \Ch_q(V(\lambda))=q^{s_\nu+s_\mu-t_{\ndm}}/[\nu]![\mu]!\chi_\lambda^{\bfi_\nu\bfi_\mu}
\end{align} 
where $t_{\ndm}$ is equal to $\rdim \E_\nu+\rdim \E_\mu$ and $\chi_{\lam}^{\bfi_\nu\bfi_\mu}$ is the coefficient of the word $\bfi_\nu\bfi_\mu$ in the expression of the dual canonical base $b_{\lam}^*$. Here $s_\nu+s_\mu$ and $[\nu]![\mu]!$ are given in~(\ref{setting of vector}).
\end{Theorem}
\begin{proof}
Since the Lusztig's sheaf $L_{\bfi_\nu\bfi_\mu}$ corresponds to the element $\theta_{\bfi_\nu\bfi_\mu}$ by Theorem~\ref{theom3.2} and the element $\theta_{\bfi_\nu\bfi_\mu}$ is expressed in terms of canonical bases as follows: 
\begin{align}\label{decomposition1.1}
 \theta_{\bfi_\nu\bfi_\mu}=\sum_{\lambda \in \cP^{\npm}} f_{\ndm}^\lambda (q) b_\lambda 
\end{align}
 we have  
\begin{align}\label{computation coefficients}
 f_{\ndm}^\lambda (q) &= (\theta_{\bfi_\nu\bfi_\mu}, b_\lambda^*)=\chi_\lambda^{\bfi_\nu\bfi_\mu}
\end{align} 
By the definition \ref{Ch}. Moreover
\begin{equation}\label{computation coefficients 1.2}
\begin{aligned}
    p_!\II_{\E_\ndm}&=\II_\nu\star\II_\mu[-\rdim \E_\nu-\rdim \E_\mu]\\
    &= q^{s_\nu+s_\mu-t_{\ndm}}L_{\bfi_\nu\bfi_\mu}/[\nu]![\mu]!   
\end{aligned} 
\qquad  
\begin{aligned}
 &\text{by the equation~(\ref{3.2})}\\
 &\text{by the equation~(\ref{reduction 1})}
\end{aligned}       
\end{equation}
 By Theorem~\ref{theom3.2}, one writes $\Ch_q(p_!\II_{\E_\ndm})\in \bff_\cA$ for the element corresponding to the complex sheaf $p_!\II_{\E_\ndm}$. We have
 \begin{align*}
  \Ch_q(p_!\II_{\E_\ndm})&=\sum_{\lambda\in \cP^\npm}\Ch_q(V(\lambda))(q^{-1})b_\lambda &\text{by Lemma~\ref{BBD decomposition} and Theorem~\ref{theom3.2}} \\
   &=q^{t_{\ndm}-s_\nu-s_\mu}\theta_{\bfi_\nu\bfi_\mu}/[\nu]![\mu]!  &\text{by the equation~(\ref{computation coefficients 1.2})} \\               
   &=q^{t_{\ndm}-s_\nu-s_\mu}/[\nu]![\mu]! \left(\sum_{\lambda \in \cP^{\npm}} f_{\ndm}^\lambda b_\lambda\right) &\text{by the equation~(\ref{decomposition1.1})}
 \end{align*} 

Thus we conclude 
\begin{align*}
 \Ch_q(V(\lambda)) =& q^{s_\nu+s_\mu-t_{\ndm}}/[\nu]![\mu]!f_{\ndm}^\lambda (q^{-1})\\ 
  &=q^{s_\nu+s_\mu-t_{\ndm}}/[\nu]![\mu]!\chi_\lambda^{\bfi_\nu\bfi_\mu}
\end{align*}   
By the equation~(\ref{computation coefficients}) and $\chi_\lambda^{\bfi_\nu\bfi_\mu}(q)=\chi_\lambda^{\bfi_\nu\bfi_\mu}(q^{-1})$               
\end{proof}
\begin{Corollary}[Support theorem]
Under the same assumption in Theorem~\ref{theorem of computation coefficients}. For the map $p:\E_{\nu,\mu}\to \E_{\nu+\mu}$ and the decomposition $p_!\mathbb{I}_{\E_{\nu,\mu}}=\bigoplus_{\lambda\in \cP^\npm}V(\lambda)\boxtimes \IC(\lambda)$, we have $V(\lam)\neq 0$ if and only if $\bfi_\nu\bfi_\mu \in supp(b_\lam^*)=\{\bfk\in \la I\ra_{\npm}|\, \chi_\lam^{\bfk}\neq 0\}$, where $b_\lambda^*=\sum\limits_{\scriptscriptstyle{\bfk\in \la I\ra_{\npm}}}\chi_\lambda^\bfk \bfk$
\end{Corollary}
\begin{proof}
Since $V(\lam)\neq 0$ if and only if $\Ch_q(V(\lam))\neq 0$. By Theorem~\ref{theorem of computation coefficients}, it implies that $\Ch_q(V(\lam))\neq 0$ if and only if $\chi_\lambda^{\bfi_\nu\bfi_\mu}\neq 0$, that is $\bfi_\nu\bfi_\mu\in supp(b_\lam^*)$. 
\end{proof} 

\begin{Remark}
In~\cite{FR}, they studied the support of the degeneration flag variety. Our results gives another description of the support theorem for any acyclic quiver with any dimension vector.
\end{Remark}

\begin{Theorem}\label{cohomology of quiver grassmannians}
Under the same assumption in Theorem~\ref{theorem of computation coefficients}. Let $M$ be a representation of an acyclic quiver $Q$, then we have 
\begin{align}\label{cohomology of quiver grassmannians 1}
  \Ch_q(\Gr_\mu(M))=q^{s_\nu+s_\mu-t_{\ndm}-\rdim \E_\ndm}/[\nu]![\mu]!\sum_{\lambda \in \cP^{\npm}}\chi_\lambda^{\bfi_\nu\bfi_\mu}\Ch_q(i_M^*\IC(\lambda)) 
\end{align}
\end{Theorem}

\begin{proof}
Recall that $q[L]=[L[1]]$ for an object in the derived category $\mathcal{D}_{\GL}^b(\E_\npm)$, (see \ref{liftq})
\begin{align*}
 \Ch_q(\Gr_\mu(M)) &= q^{-\rdim \E_\ndm}\Ch_q(i_M^*p_!\II_{\E_\ndm}), &\text{by Lemma~\ref{lem5.2}}\\
  &=q^{-\rdim \E_\ndm}\sum_{\lambda \in \cP^{\npm}} \Ch_q(V(\lambda))\Ch_q(i_M^*\IC(\lambda))  &\text{by Lemma~\ref{BBD decomposition}}\\
  &=q^{s_\nu+s_\mu-t_{\ndm}-\rdim \E_\ndm}/[\nu]![\mu]!\sum_{\lambda \in \cP^{\npm}}\chi_\lambda^{\bfi_\nu\bfi_\mu}\Ch_q(i_M^*\IC(\lambda))  &\text{by Theorem~\ref{theorem of computation coefficients}}
\end{align*}
\renewcommand{\qed}{}
\end{proof}

\begin{Definition}
For an acyclic quiver $Q$, a representation $M$ with dimension $\nu$ is called a rigid representation if and only if its $\GL(\nu)$-orbit, denoted by $\cO_M$, is open in $\E_{\nu}$. This definition is equivalent to the condition $\Ext_Q^1(M,M)=0$.
\end{Definition}
\begin{Theorem}\label{cohomology of rigid quiver grassmannians}
Let $M$ be a rigid representation of $Q$ with dimension $\undd(M)=\npm$. For the canonical base $\theta_\npm=\theta_{i_1}^{((\npm)_1)}\cdots \theta_{i_n}^{((\npm)_n)}$, see Lemma~\ref{lemma schi}, we express its dual base as $\theta_\npm^*=b_{(\npm)^0}^*=b_{\bfi(\npm)}^*$ for a unique word $\bfi(\npm)$ by Remark~\ref{dual canonical bases}. Hence we have that 
\begin{align}\label{cohomology of rigid quiver grassmannians 1.1}
  \Ch_q(\Gr_\mu(M))=q^{s_\nu+s_\mu-t_{\ndm}-\la\mu,\nu\ra}/[\nu]![\mu]!\chi_{\bfi(\npm)}^{\bfi_\nu\bfi_\mu}
\end{align}
where $t_{\ndm}$ is equal to $\rdim \E_\nu+\rdim \E_\mu$, $\chi_{\bfi(\npm)}^{\bfi_\nu\bfi_\mu}$ is the coefficient of the word $\bfi_\nu\bfi_\mu$ in the expression of the dual canonical base $b_{\bfi(\npm)}^*$, and $\la\mu,\nu\ra$ refers to the Euler form~\ref{2.1}. Here $s_\nu+s_\mu$ and $[\nu]![\mu]!$ are given in~(\ref{setting of vector}). 
\end{Theorem}
\begin{proof}
We claim that there is only the element $(\npm)^0\in \cP^\npm$ such that \[i_M^*(\IC((\npm)^0))\neq \emptyset\]

Suppose a simple perverse sheaf $\IC(Y,\lambda)$ satisfies $M\in Supp(\IC(Y,\lambda))$, where $Supp(\IC(Y,\lambda))$ stands for the support of the complex $\IC(Y,\lambda)$. Since $Supp(\IC(Y,\lambda))\subset \overline{Y}$ and $Y$ is a $GL(\npm)$-invariant subvariety, it implies that $\cO_M\subset \overline{Y}$. 

Since the orbit $\cO_M$ is an open subvariety of $\E_\npm$ and affine space $\E_\npm$ is irreducible, it implies $\overline{Y}=\E_\npm$. Meanwhile, since $Y$ is open in $\E_\npm$, it leads to that $\cO_M\cap Y\neq \emptyset$. For $Y$ is a $GL(\npm)$-invariant subvariety, we have $\cO_M\subset Y$ and $\cO_M$ is an open subset of $Y$. 

Denote the open embedding of $Y$ by $j:Y\hookrightarrow \E_\npm$. By the lemma-Definition 5.2 of~\cite{KW}, we have that $\IC(Y,\lambda)=j_{!*}A$ where $A$ is a simple perverse sheaf on $Y$. If $A$ admits a perverse subsheaf arising from $Y/\cO_M$, then one has $Supp(\IC(Y,\lambda))\subset \E_\npm/\cO_M$ for $\IC(Y,\lambda)$ is a simple perverse sheaf. It contradicts the assumption. Hence it implies $A=j_{!*}^\prime B$ where $j^\prime: \cO_M\hookrightarrow Y$ and $B$ is a simple perverse sheaf on $\cO_M$. Thus it follows $\IC(Y,\lambda)=j_{!*}\circ j_{!*}^\prime B=\IC(\cO_M)$ for there exists only one irreducible local system $\CC_{\cO_M}$ on $\cO_M$. By Lemma~\ref{exam3.1}, one obtains \[\IC((\npm)^0)=\IC(\cO_M)=\II_{\E_\npm}=\CC_{\rdim \E_\npm}[\rdim \E_\npm] \]  Hence the unique element $\lam$ of $\cP^\npm$ such that $i_M^*(\IC(\lam))\neq \emptyset$ is the element $(\npm)^0$. 
It is shown that the canonical base corresponding to $(\npm)^0$ is equal to $\theta_\npm$ by Lemma~\ref{lemma schi}, thus the equation~(\ref{cohomology of quiver grassmannians 1}) becomes 
\begin{align*}
 \Ch_q(\Gr_\mu(M))&=q^{s_\nu+s_\mu-t_{\ndm}-\rdim \E_\ndm}/[\nu]![\mu]!\chi_{\bfi(\npm)}^{\bfi_\nu\bfi_\mu}\Ch_q(\IC((\npm)^0))\\
 &=q^{s_\nu+s_\mu-t_{\ndm}-\rdim \E_\ndm}/[\nu]![\mu]!\chi_{\bfi(\npm)}^{\bfi_\nu\bfi_\mu}\Ch_q(\CC_{\rdim \E_\npm}[\rdim \E_\npm])\\
 &=q^{s_\nu+s_\mu-t_{\ndm}-\rdim \E_\ndm}/[\nu]![\mu]!\chi_{\bfi(\npm)}^{\bfi_\nu\bfi_\mu}q^{\rdim \E_\npm}\\
 &=q^{s_\nu+s_\mu-t_{\ndm}-\rdim \E_\ndm+\rdim \E_\npm}/[\nu]![\mu]!\chi_{\bfi(\npm)}^{\bfi_\nu\bfi_\mu}\\
 &=q^{s_\nu+s_\mu-t_{\ndm}-\la\mu,\nu\ra}/[\nu]![\mu]!\chi_{\bfi(\npm)}^{\bfi_\nu\bfi_\mu}
\end{align*}
By the dimension formula~\ref{dimension formula}
\end{proof}

Next Corollary seems quite trivial(also by \emph{Grauert's semicontinuity theorem}), but it indeed reveals some additional relationship between quiver Grassmannians of different representations. 
\begin{Corollary}\label{cohomology of quiver subgrassmannians }
Suppose that $N\ncong M$ where $M$ is a rigid representation such that $\rdim N=\rdim M$, we have that
\begin{align}\label{cohomology of quiver subgrassmannians 1.1}
 \Ch_q(\Gr_\mu(N))=\Ch_q(\Gr_\mu(M))+ f(q)
\end{align}
where $f(q)\in \mathbb{N}[q,q^{-1}]$.
\end{Corollary}
\begin{proof}
Set $q^{Gr_\mu^{\npm}}\defeq q^{s_\nu+s_\mu-t_{\ndm}-\rdim \E_\ndm}/[\nu]![\mu]!$. By the equation~\ref{cohomology of quiver grassmannians 1}, it induces 
\begin{align*}
 \Ch_q(\Gr_\mu(N))&=q^{Gr_\mu^{\npm}}\sum_{\lambda \in \cP^{\npm}}\chi_\lambda^{\bfi_\nu\bfi_\mu}\Ch_q(i_N^*\IC(\lambda))\\
&=q^{Gr_\mu^{\npm}}\sum_{\lambda \neq (\npm)^0}\chi_\lambda^{\bfi_\nu\bfi_\mu}\Ch_q(i_N^*\IC(\lambda))+ q^{Gr_\mu^{\npm}} \chi_{\bfi(\npm)}^{\bfi_\nu\bfi_\mu}\Ch_q(i_N^*\II_{\E_\npm})\\
&=q^{Gr_\mu^{\npm}}\sum_{\lambda \neq (\npm)^0}\chi_\lambda^{\bfi_\nu\bfi_\mu}\Ch_q(i_N^*\IC(\lambda))+\Ch_q(M)
\end{align*}
for $\mathbb{I}_{\E_{\nu+\mu}}$ is the constant perverse sheaf. Thus we conclude 
\[
\Ch_q(\Gr_\mu(N))=\Ch_q(\Gr_\mu(M))+ f(q)
\]
where $f(q)\in \mathbb{N}[q,q^{-1}]$.
\end{proof} 

\begin{Remark}\label{rem:lanni}
By the paper~\cite{LS}, Professor Cerulli Irelli proposed an open problem: Compare the cohomology of  two quiver Grassmannians $\HH^\bullet \Gr_\mu(M)$ and $\HH^\bullet Gr_\mu(N)$ such that the representation $M$ is a degeneration of the representation $N$. For the definition of the degeneration, we can refer to~\cite{Scho} or~\cite{CB}. By the Geometric representation theory, these two Poincar\'e polynomials are two dual standard modules.     
\end{Remark}

By~\cite[Corollary 2]{IEFR}, we have rigid quiver Grassmannians has polynomial point-count. Indeed, Fan Qin has proved this fact in~\cite{Qin}. Hence, we can calculate the number of points of rigid quiver Grassmannians over a finite field $\FF_q$ by Theorem~\ref{cohomology of rigid quiver grassmannians}. 
\begin{Corollary}
Let us fix a finite field $\FF_q$ and its algebraically closed field $\overline{\mathbb{F}_q}$. Set $q_*=q^{\frac{1}{2}}$. For a rigid representation $M$ over $\FF_q$ as Theorem~\ref{cohomology of rigid quiver grassmannians}, it follows 
\[
\mid \Gr_\mu(M)(\mathbb{F}_{q})\mid=q_*^{s_\nu+s_\mu-t_{\ndm}-\la\mu,\nu\ra}/[\nu]![\mu]!\chi_{\bfi(\npm)}^{\bfi_\nu\bfi_\mu}(q_*)
\]
we refer to Theorem~\ref{cohomology of rigid quiver grassmannians} for the notations here.  
\end{Corollary}
\begin{proof}
By~\cite[Theorem 1 and Corollary 2]{IEFR}, we have 
\begin{align*}
\mid \Gr_\mu(M)(\mathbb{F}_{q})\mid &=\mathop{\sum}\limits_{i}\rdim \HH^{2i}(\Gr_\mu(M))q^i \\
 &=\Ch_{q_*}(\HH^\bullet \Gr_\mu(M)) \qquad \text{for $ \Gr_\mu(M)$ has no odd cohomology}\\
 &=q_*^{s_\nu+s_\mu-t_{\ndm}-\la\mu,\nu\ra}/[\nu]![\mu]!\chi_{\bfi(\npm)}^{\bfi_\nu\bfi_\mu}(q_*)
\end{align*}
\end{proof}

\section{Another proof of odd cohomology vanishing for quiver Grassmannians}\label{antherproof}

In this section, we will give a new proof of the odd cohomology vanishing theorem for quiver Grassmannians. By Theorem~\ref{cohomology of rigid quiver grassmannians}, we just need prove the coefficients $\chi_\lambda^{\bfinm}$ of the dual canonical bases $b_\lambda^*$ satisfy 
\begin{equation}\label{odd}
	\chi_\lambda^{\bfinm}\in \ZZ[q^2,q^{-2}] \quad \text{or}\quad \chi_\lambda^{\bfinm}\in q\ZZ[q^2,q^{-2}]
\end{equation}
It is known that the graded module category of Quiver Hecke algebras is the categorifacation of the algebra $\cA$, and the simple modules correspond to the dual canonical bases by the Character map. Thus we recall some notions about quiver Hecke algebras.

A quiver Hecke algebra $H_\nu$ for a $I$-graded element $\nu$ is  a associative algebra generated by $\{1_\bfi|\, \bfi\in \la I\ra_\nu\}\cup\{x_1,\cdots,x_n\}\cup\{\tau_1,\cdots,\tau_{n-1} \}$ subject to some relations, where $n=ht(\nu)$ (see~\cite{KL09} for more details). We denote by $\Rep$ the category of its graded finite dimensional modules, then we write $\K_0(\Rep)$ for the Grothendieck group of the category $\Rep$. For a graded vector space $V=\oplus_{n\in\ZZ} V_n$, Its graded dimension is $\Dim V\defeq \sum_{n\in \ZZ}(\rdim V_n)q^n$. The \emph{Character} of a finite dimensional module $V$ is defined by
\begin{equation}
	\Ch(V)\defeq \sum_{\bfi\in\la I\ra_\nu}(\Dim 1_\bfi V)\bfi
\end{equation}
Set $\mathbf{B}^*\defeq \{[L]|\, \text{for all self-dual irreducible $H$-modules $L$}\}$ (see~\cite[Section 3]{B}), we have
\begin{Theorem}\cite{VV}\cite[Theorem 3.11]{B}
the Character map $\Ch:\K_0(\Rep)\to \bff_\cA^* \subset \ZZ[q,q^{-1}]\la I\ra_\nu$ is isomorphic, and send $\mathbf{B}^*$ to the dual canonical basis.
\end{Theorem}
A useful lemma is given as follows

\begin{Lemma}\cite[Lemma 3.2]{B}\label{decomposition}
Fix a total order $<$ on $I$. For a word $\bfi \in\langle I\rangle$ of length $n$ define
$$
p(\bfi):=\sum_{\substack{1 \leq j<k \leq n\\ i_j<i_k}} \alpha_{i_{j}} \cdot \alpha_{i_{k}} \quad(\bmod 2)
$$
Then every $H_\nu$-module $V$ decomposes as a direct sum of modules as $V=V^{\overline{0}} \oplus V^{\overline{1}}$ where
$$
V^{q}:=\bigoplus_{\substack {\bfi \in\langle I\rangle, n \in \mathbb{Z} \\ n \equiv p(\bfi)+q(\bmod 2)}} 1_{\bfi} V_{n}
$$
\end{Lemma}

\begin{Lemma}
Under the assumption of Theorem~\ref{cohomology of rigid quiver grassmannians}, we have 
\begin{equation}
	p(\bfi_\nu\bfi_\mu)\equiv t_{\ndm}+\sum_{h\in \Omega} \nu_{s(h)}\mu_{t(h)}(\bmod\, 2)
\end{equation}
\end{Lemma}
\begin{proof}
For $\bfi_\nu\bfi_\mu=[i_1,\cdots,i_1,i_2,\cdots,i_n,\cdots,i_n,i_1,\cdots,i_1,\cdots,i_n]$, we have
\begin{equation}
p(\bfi_\nu\bfi_\mu)\equiv \sum_{i_k>i_l} \nu_{l}(\npm)_{k}\, \alpha_{i_{k}} \cdot \alpha_{i_{l}}+\sum_{i_k>i_l} \mu_l\mu_k\, \alpha_{i_{k}} \cdot \alpha_{i_{l}}\quad(\bmod 2)
\end{equation}

Since the quiver $Q$ is acyclic and $i_k>i_l$ if and only if there exists an arrow $i_l\to i_k$. we have $\alpha_{i_{k}} \cdot \alpha_{i_{l}}=\sharp \{\text{arrows from $i_l\to i_k$}\}$ for any $i_l<i_k$. It follows
\begin{equation}
\begin{split}
p(\bfi_\nu\bfi_\mu)&\equiv \sum_{h\in \Omega} \nu_{s(h)}(\npm)_{t(h)}+\sum_{h\in \Omega} \mu_{s(h)}\mu_{t(h)} \quad(\bmod 2)\\
	&\equiv t_{\ndm}+\sum_{h\in \Omega} \nu_{s(h)}\mu_{t(h)}\quad (\bmod\, 2)
\end{split}
\end{equation}
By the equation $t_{\ndm}=\sum_{h\in \Omega} \nu_{s(h)}\nu_{t(h)}+\sum_{h\in \Omega} \mu_{s(h)}\mu_{t(h)}$

\end{proof}
\begin{Theorem}
Under the assumption of Theorem~\ref{cohomology of rigid quiver grassmannians}. The Poinca\'e polynomial of rigid quiver Grassmannian $\Ch_q(\Gr_\mu(M))$ satisfies
\begin{equation}
	\Ch_q(\Gr_\mu(M))\in\ZZ[q^2]
\end{equation}
In other words, $\HH^{2i+1}(\Gr_\mu(M))=0$ for all $i\in \NN$.
\end{Theorem}
\begin{proof}
For the dual canonical base $b_{\bfinm}^*$, the corresponding simple module is denoted by $L(\bfinm)$. that means that $\Ch(L(\bfinm))=b_{\bfinm}^*$. By the definition of Character map, we have 
\begin{equation}\label{eq:coefficient}
 	\chi_{\bfinm}^{\bfi_\nu\bfi_\mu}(q)=\Dim 1_{\bfi_\nu\bfi_\mu} L(\bfinm)=\sum_{n\in\ZZ} \rdim(1_{\bfi_\nu\bfi_\mu} L(\bfinm)_n )q^n
\end{equation}
Since the module $L(\bfinm)$ is a simple module, we have $L(\bfinm)=L(\bfinm)^{\overline{0}}$ or $L(\bfinm)=L(\bfinm)^{\overline{1}}$ by Lemma~\ref{decomposition}. 

If $L(\bfinm)=L(\bfinm)^{\overline{0}}$, we have $\rdim(1_{\bfi_\nu\bfi_\mu} L(\bfinm)_n )\neq 0$ if and only if 
\begin{equation}\label{eq:mod}
	n\equiv p(\bfi_\nu\bfi_\mu)(\bmod 2)
\end{equation}
 By Equation (\ref{element setting 1}) we have 
$$q^{-s_\nu-s_\mu}[\nu]![\mu]!=\sum_{w\in S_\nu\times S_\mu}q^{-2l(w)}$$
Hence $q^{s_\nu+s_\mu}([\nu]![\mu]!)^{-1}\in \QQ(q^2)$. By the equations (\ref{eq:mod})and (\ref{eq:coefficient}), we have 
\begin{equation}
	q^{-t_{\ndm}-\la\mu,\nu\ra}\chi_{\bfinm}^{\bfi_\nu\bfi_\mu}(q)=q^{-t_{\ndm}-\la\mu,\nu\ra+p(\bfi_\nu\bfi_\mu)}h(q^2)
\end{equation}
where $h(q^2)\in \ZZ{[q^2,q^{-2}]}$. We clam that $q^{-t_{\ndm}-\la\mu,\nu\ra}\chi_{\bfinm}^{\bfi_\nu\bfi_\mu}(q)\in \ZZ{[q^2,q^{-2}]}$. Otherwise, we have $q^{-t_{\ndm}-\la\mu,\nu\ra}\chi_{\bfinm}^{\bfi_\nu\bfi_\mu}(q)\in q\ZZ{[q^2,q^{-2}]}$, it follows 
\begin{equation}
\Ch_q(\Gr_\mu(M))=q^{s_\nu+s_\mu}([\nu]![\mu]!)^{-1}q^{-t_{\ndm}-\la\mu,\nu\ra}\chi_{\bfinm}^{\bfi_\nu\bfi_\mu}(q)\in q\QQ(q^2)
\end{equation}
Since $\Ch_q(\Gr_\mu(M))\in \ZZ[q]$ and $\ZZ[q]\cap q\QQ(q^2)=q\ZZ[q^2]$, It implies that $\Ch_q(\Gr_\mu(M))$ has no constant item. In other words, $\HH^0(\Gr_\mu(M))=0$. Since $\Gr_\mu(M)$ is smooth, the algebra cycle $[\Gr_\mu(M)]\in \HH^0(\Gr_\mu(M))$. It leads to a contradiction. Hence we have $q^{-t_{\ndm}-\la\mu,\nu\ra}\chi_{\bfinm}^{\bfi_\nu\bfi_\mu}(q)\in \ZZ{[q^2,q^{-2}]}$, By above discussion we have $\Ch_q(\Gr_\mu(M))\in \QQ(q^2)\cap \ZZ[q]$. Since $\QQ(q^2)\cap \ZZ[q]=\ZZ[q^2]$, we have $\Ch_q(\Gr_\mu(M))\in \ZZ[q^2]$

If $L(\bfinm)=L(\bfinm)^{\overline{1}}$, we have $\rdim(1_{\bfi_\nu\bfi_\mu} L(\bfinm)_n )\neq 0$ if and only if $$n\equiv p(\bfi_\nu\bfi_\mu)+1(\bmod 2)$$ Similarly, we also have  $q^{-t_{\ndm}-\la\mu,\nu\ra}\chi_{\bfinm}^{\bfi_\nu\bfi_\mu}(q)\in \ZZ{[q^2,q^{-2}]}$ and  $\Ch_q(\Gr_\mu(M))\in \QQ(q^2)\cap \ZZ[q]$. Thus we have done.
\end{proof}

\section{The cohomology of quiver Grassmannians for Dynkin quivers}\label{Dynkin}
In this section, we fix an orientation $\Omega$ of a Dynkin quiver $Q$. It is known that there exists a reduced expression of the longest element $w_0$ adapting the orientation $\Omega$, which gives rise to a convex order on $R^+$, denoted by $>$.
We set another convex order as
\begin{align}\label{ordering}
\beta_k\prec\beta_l \qquad  \text{for $\beta_k>\beta_l$ where $\beta_k, \beta_l\in R^+$}
\end{align}
By Gabriel's Theorem, we have
\begin{align}\label{5.1}
\Hom_{\overrightarrow{Q}}(M(\beta_a),M(\beta_b))=\Ext_{\overrightarrow{Q}}(M(\beta_b),M(\beta_a))=0 \qquad \text{for $\beta_a\prec\beta_b$}
\end{align}
where $M(\beta_i)$ stands for the indecomposable module with dimension $\beta_i$\\

Next we will find out this ordering on the simple modules satisfies the condition of the ordering of the set of the vertices $I$ in~\S\ref{first}.  Denote the simple roots ( viewed as the 
vertices) by $\{i_k\}_{1\leq k\leq n}$.
 If there exists an arrow going from $i_k$ to $i_l$, then by the equations (\ref{2.1}) and (\ref{2.2}), it follows
\[
\begin{split}
\langle i_k,i_l\rangle &= -\mathop{\sum}\limits_{\mathop{h\in \Omega}\limits_{s(h)=i_k,t(h)=i_l}}1\\
         &=[M(i_k),M(i_l)]-[M(i_k),M(i_l)]^1\\
         &=-[M(i_k),M(i_l)]^1
\end{split}
\]
By the equation (\ref{5.1}), we have that $i_k\prec i_l$. Thus the ordering of $I$ in~\S\ref{first} coincides with the above order $\prec$ restricting on $I$.  We will calculate the coefficients $\chi_{\bfi(\npm)}^{\bfi_\nu\bfi_\mu}$ of dual canonical bases $b_{\bfi(\npm)}^*$ for Dynkin quivers in terms of dual PBW basis.

\subsection{A useful Lemma}
A kind of special words, called a Lyndon word, is used to study the dual PBW basis of quantized enveloping algebra for Dynkin quivers. Suppose we have defined an order on the set of the vertices $I$, and we define a order of its words by the lexicographic order, see~\ref{shuffle algebra}.    
\begin{Definition}
A word $\bfj=[j_1\cdots j_m]\in \langle I\rangle$ is called Lyndon word, if it satisfies the condition
\[
[j_1\cdots j_m]< [j_l\cdots j_k]\qquad \text{for all $1\leq l< k\leq m$}
\]
\end{Definition}

For Cartan data of finite type (such as Dynkin quivers), the good Lyndon words can be classified by the following fact in~\cite{LR}.
\begin{Lemma}\cite[Proposition 3.2]{LR}\label{lemm3.2}
If the Cartan data is of finite type then the map $\bfi\mapsto |\bfi|$ is a bijection between the set of good Lyndon words and the set $R^+$ of the positive roots. Here $|\bfi|$ is given in~(\ref{setting shuffle})
\end{Lemma}
Thus the above bijection transforms the lexicographic order on good Lyndon words to a total order on the set $R^+$, which is also denoted by $\prec$. 

Since for any good word $\bfi\in \langle I\rangle^+$, we can decompose it as $\bfi=\textbf{j}_1\cdots\textbf{j}_k$ where $\textbf{j}_l$ are good Lyndon words with the condition $\textbf{j}_1\geq\textbf{j}_2\cdots\geq\textbf{j}_k$, thus any good word $\bfi$ can be expressed as the form $(\beta_1,\beta_2,\cdots,\beta_m)$ such that $\beta_1\succeq\beta_2\succeq\cdots\succeq\beta_m$ by Theorem~\ref{lemm3.2}. Hence, we can identify the following set 
\[\KP(\nu)\defeq \{(\beta_1,\beta_2,\cdots,\beta_m)| \sum_{i=1}^m\beta_i=\nu \ \text{such that $\beta_1\succeq\beta_2\succeq\cdots\succeq\beta_m$}\}\] 
with the set of good words $\la I\ra_\nu^+$. It is well known that the set $\KP(\nu)$ coincides with the set of the $\GL(\nu)$-orbits in the space $\E_\nu(Q)$, see~\S\ref{Lusztig's construction}.
\begin{Remark}\label{rem:note}
Note this order usually doesn't coincide with the original order at the beginning of this section, for the order of positive roots depends on the orientation $\Omega$ and the induced lexicographic order on $R^+$ depends only on the order of $I$.
\end{Remark}

\subsection{Dual PBW basis for Dynkin quivers}
Return back our assumption of the beginning of this section. In~\cite{BKM}, a \emph{Kostant partition} of $\nu\in \NN[I]$ is a sequence $\lambda=(\lambda_1,\cdots,\lambda_l)$ of the positive roots such that $\lambda_1\succeq\cdots\succeq\lambda_l$ and $\lambda_1+\cdots+\lambda_l=\nu$. We can define an order $\succ$ on $\KP(\nu)$ so that $\lambda\succ\kappa$ if and only if both of the following hold
\begin{center}
$\lambda_1=\kappa_1,\cdots,\lambda_{k-1}=\kappa_{k-1}$ and $\lambda_k\succ\kappa_k$ for some $k$ such that $\lambda_k$ and $\kappa_k$ are both defined.
\end{center}
\begin{center}
  $\lambda_1^\prime=\kappa_1^\prime,\cdots,\lambda_{k-1}^\prime=\kappa_{k-1}^\prime$ and $\lambda_k^\prime\prec\kappa_k^\prime$ for some $k$ such that $\lambda_k^\prime$ and $\kappa_k^\prime$ are both defined.
\end{center}
where $\lambda_k^\prime=\lambda_{l+1-k}$.

For $\lambda=(\lambda_1,\cdots,\lambda_l)\in \KP(\nu)$, we set
\begin{align}\label{symbol partition}
s_{\lambda}\defeq \mathop{\sum}\limits_{\beta\in R^+}\frac{1}{2}m_{\beta}(\lambda)(m_{\beta}(\lambda)-1)
\end{align}
where $m_{\beta}(\lam)$ denotes the multiplicity of the root $\beta$ in $\lambda$.

\begin{Definition}
For any root $\beta\in R^+$, we have a Lyndon good word corresponding to it, denoted by $\bfj(\beta)$. By Remark~\ref{dual canonical bases}, there exists a dual canonical base $b_\beta^*$. Given an element $\lam=(\lam_1,\lam_2,\cdots,\lam_l)\in \KP(\nu)$, we define its dual PBW base as 
\[
 r_\lam^*\defeq q^{s_\lam} b_{\lam_1}^*\circ b_{\lam_2}^*\circ\cdots \circ b_{\lam_l}^*
\]
where $\circ$ refers to the shuffle product, see~(\ref{shuffle product}).
\end{Definition}
\begin{Remark}\label{PBW and canonical}
For an element $\lam\in \KP(\nu)$, we have:
\begin{align}
b_\lam^*=r_\lam^*+\sum_{\kappa\prec \lam}f_\lam^{\kappa}(q)r_\kappa^*
\end{align} 
where $f_\lam^{\kappa}(q)\in q\ZZ[q]$.
\end{Remark}

\subsection{The cohomology of quiver Grassmannians of Dynkin quivers}
In this section, we focus on a special orientation of a Dynkin quiver.\\

\noindent
{\bf Assumption }: As a partial order on $\KP(\nu)$ doesn't always arise from the lexicographic ordering $\la I\ra_\nu^+$, thus from then on, we assume the orientation $\Omega$ gives rise to an order on $\KP(\nu)$ coinciding with the total order $\la I\ra_\nu^+$ via Lemma~\ref{lemm3.2} or see Remark~\ref{rem:note}, that is: the ordering of Kostant partition $\KP(\nu)$ coincides with the total lexicographic ordering on $\la I\ra_\nu$ via their isomorphism: $\bfi\to |\bfi|$.\\

By the definition of $\KP(\nu)$ and the Gabriel's theorem, we have the set of the isomorphism classes of representations on a quiver $Q$ with $\undd (M)=\nu$ is equal to the set $\KP(\nu)$, and then there is a bijection between $\KP(\nu)$ and the set of the $\GL(\nu)$-orbits in $\E_\nu(\overrightarrow{Q})$,  denoted by $|\E_\nu/\GL(\nu)|$, given by 
\[
\begin{split}
\KP(\nu)\xrightarrow{\thicksim} |\E_\nu/\GL(\nu)|\\
(\lambda_1,\cdots,\lambda_k)\mapsto \mathop{\bigoplus}\limits_{1\leq i\leq k}[M(\lambda_i)]
\end{split}
\]
where $[M(\lambda_i)]$ denotes the isomorphism class of the indecomposable module $M(\lambda_i)$ corresponding to the root $\lambda_i$. Hence for an element $\lambda\in \KP(\nu)$, we denote by $\cO_\lambda$ the corresponding $\GL(\nu)$--orbit and by $M(\lambda)$ the corresponding representation.

Next we will find the relation between the degeneration of representations on $Q$ and the order on $\KP(\nu)$. We recall another partial order on the set $\KP(\nu)$: Given two elements $\lambda,\kappa\in \KP(\nu)$, we define a partial order $\prec^1$ so that $\lambda\prec^1\kappa$ if and only if we have $\mathcal{O}_\kappa\subset \overline{\cO_\lambda}$. Denote by $\nu^0$ the Konstant partition corresponding to the unique open orbit in $\E_\nu(\overrightarrow{Q})$. Note that representation $M(\nu^0)$ is a rigid representation, thus the notation here coincides with the notation in Lemma~\ref{lemma schi} by Theorem~\ref{cohomology of rigid quiver grassmannians}, so it leads to no confusion. 

For the next theorem, we recall some facts on the partial order $\prec^1$. In~\cite{CB}, for $\overrightarrow{Q}$ is a Dynkin quiver, the condition $\lambda\prec^1\kappa$ implies that $[N,M(\lambda)]\leq[N,M(\kappa)]$ and  $[M(\lambda),N]\leq[M(\kappa),N]$ for all representations $N$ of the quiver $\overrightarrow{Q}$, see section~\S\ref{first} for the notations. 
\begin{Theorem}\label{smallest element in KP}
The partition $\nu^0$ is a smallest element of $\KP(\nu)$ under the ordering $\prec$.
\end{Theorem}
\begin{proof}
	 Suppose for the contradiction that there exists an element $\lambda\prec \nu^0$ in $\KP(\nu)$. By the definition of this partial order, we have that $\lambda_k \prec\nu_k^0$ for some $k$ such that $\lambda_i =\nu_i^0$ for all $1\leq i\leq k-1$. Let $m$ be the multiplicity of $\nu_k^0$ in the sequence $\nu^0$. By the equation~(\ref{5.1}) and the definition of Kostant partitions, taking the module $M(\nu_k^0)$, it implies
\begin{align*}
   [M(\nu^0),M(\nu_k^0)]&=[M(\nu_k^0)^m\oplus(\oplus_{i\leq k-1}M(\nu_i^0)),M(\nu_k^0)] \\
   &=m[M(\nu_k^0),M(\nu_k^0)]+[(\oplus_{i\leq k-1}M(\nu_i^0)),M(\nu_k^0)] \\
   &=m+[(\oplus_{i\leq k-1}M(\nu_i^0)),M(\nu_k^0)]&\text{by  $[M(\nu_k^0),M(\nu_k^0)]=1$}\\
   &=m+[(\oplus_{i\leq k-1}M(\lambda_i)),M(\nu_k^0)]  &\text{by $\lambda_i =\nu_i^0$ for all $1\leq i\leq k-1$}\\
   &=m+[M(\lambda)),M(\nu_k^0)]& \text{by $\lambda_i \prec\nu_k^0$ for all $ i\geq k$}
\end{align*}
 So it contradicts the facts $[M(\nu^0),N]\leq[M(\lambda),N]$ for all representations $N$ on the quiver $\overrightarrow{Q}$
\end{proof}
By Theorem~\ref{smallest element in KP} and Remark~\ref{PBW and canonical}, we have $b_{\nu^0}^*=r_{\nu^0}^*$. Suppose the Kostant partition of $\nu^0$ is $(\nu_1^0,\nu_2^0,\cdots,\nu_l^0)$, then 
\[
 r_{\nu^0}^*=q^{s_{\nu^0}}b_{\nu_1^0}^*\circ b_{\nu_2^0}^*\circ\cdots \circ b_{\nu_l^0}^*
\]
see~\ref{symbol partition} for the notation $s_{\nu^0}$. Since we can decompose a rigid module $M(\nu^0)$ as the sum of indecomposible representations, it is easy to find its Kostant partition $\nu^0$

If we write $r_{\nu^0}^*=\sum_{\bfi\in \la I\ra_\nu}h_{\nu^0}^\bfi \bfi$ for the expression of $r_{\nu^0}^*$, then we have $h_{\nu^0}^\bfi=\chi_{\nu^0}^\bfi$. Thus
\begin{Theorem}\label{cohomology of rigid quiver grassmannians Dynkin }
Given a Dynkin quiver $Q$. Let $\ndm\in \NN[I]$ be two elements and $M$ a rigid representation with dimension $\undd(M)=\npm$. then it follows  
\begin{align}\label{cohomology of rigid quiver grassmannians Dynkin 1}
  \Ch_q(\Gr_\mu(M))=q^{s_\nu+s_\mu-t_{\ndm}-\la\mu,\nu\ra}/[\nu]![\mu]!h_{(\npm)^0}^{\bfi_\nu\bfi_\mu}
 \end{align} 
 where $h_{(\npm)^0}^{\bfi_\nu\bfi_\mu}$ is the coefficient of the word $\bfi_\nu\bfi_\mu$ in the dual PBW base of $r_{(\npm)^0}^*=b_{(\npm)^0}^*$ as above. For other notations, we refer to Theorem~\ref{cohomology of rigid quiver grassmannians}.   
\end{Theorem}
\begin{Remark}
In~\cite[Propositon 56]{Lec}, all dual canonical bases corresponding to positive roots are expressed as $b_\beta^*=\sum_{\bfj\sim \bfi(\beta)}\bfj$, where the relation $\bfj\sim \bfi(\beta)$ is defined by that $\bfj$ can be obtained by a sequence of communications of two adjacent letters $i$ and $j$ with $a_{i,j}=0$. Hence, even though the calculation of the coefficients $h_{(\npm)^0}^{\bfi_\nu\bfi_\mu}$ is tough, it is possible to calculate it explicitly.   
\end{Remark}
There is a surprising corollary as follows. 
\begin{Corollary}
Under the above assumption, the map $p:\E_{\nu,\mu}\to \E_{\nu+\mu}$ is surjective, if and only if $\bfi_\nu\bfi_\mu\in supp (r_{(\npm)^0}^*)$
\end{Corollary}
\begin{proof}
Since $p:\E_{\nu,\mu}\to \E_{\nu+\mu}$ is surjective if and only if $\Gr_\mu(M)\neq \emptyset$ for a rigid representation in $\E_\npm$ (note there always exists such a representation for any element $\npm$), then it is equivalent to $\Ch_q(\Gr_\mu(M))\neq 0$. By the above theorem, it implies that $h_{(\npm)^0}^{\bfi_\nu\bfi_\mu}\neq 0$.  
\end{proof} 

\begin{Remark}
In~\cite{Scho}, Schofield used the methods of modules theory to determine the map $p$ is surjective. But the above corollary describes it in a different way. 
\end{Remark}

\subsection{Type A}
In~\cite[section~8.1]{Lec}, for any positive root $\beta$ of type $A$, its dual canonical bases $b_\beta^*$ are in the form $b_\beta^*=\bfi(\beta)$ where the word $\bfi(\beta)$ corresponding to the root $\beta$, see Lemma~\ref{lemm3.2}. It implies 
\[
 r_{\nu^0}^*=q^{s_{\nu^0}}\bfi(\nu_1^0)\circ \bfi(\nu_2^0)\circ\cdots \circ \bfi(\nu_l^0)
\]

For an element $\nu=\sum_{i_k\in I} \nu_{k} i_k$, we define its height by $\oht(\nu)\defeq \sum_{i_k\in I} \nu_k$. For the element $\nu^0$, we set $\oht(\nu^0)=n$ and $\oht(\nu_i^0)=n_i$. Recall $max(b_{\nu^0}^*)=\bfi(\nu)$ in Remark~\ref{dual canonical bases}. By the~\cite[Corollary 5.4]{KR}, we have $\bfi(\nu)=\bfi(\nu_1^0)\bfi(\nu_2^0)\cdots\bfi(\nu_l^0)$. We take the minimal representation of each element of $S_n/\prod_{i=1}^l S_{n_i}$. One can define a set by 
\[S_{\nu}^\bfj\defeq  \{w\in S_n/\prod_{i=1}^l S_{n_i}|\  w(\bfi(\nu))=\bfj\}\]
Hence by the definition of shuffle product~(\ref{shuffle product})  we have that 
\[
 \chi_{\nu^0}^\bfj=h_{\nu^0}^\bfj= q^{s_{\nu^0}} \sum_{w\in S_{\nu}^\bfj} q^{\odeg(w;\bfi(\nu))}
\]
Thus Theorem~\ref{cohomology of rigid quiver grassmannians} becomes the following form.
\begin{Theorem}\label{cohomology of rigid quiver grassmannians A }
Let $M$ be a rigid representation of a quiver $Q$ of type $A$ with dimension $\undd(M)=\npm$. then it follows  
\begin{align}\label{cohomology of rigid quiver grassmannians Dynkin 1.1}
  \Ch_q(\Gr_\mu(M))=q^{s_\nu+s_\mu+s_{(\npm)^0}-t_{\ndm}-\la\mu,\nu\ra}/[\nu]![\mu]!\sum_{w\in S_{\npm}^{\bfi_\nu\bfi_\mu}} q^{\odeg(w;\bfi((\npm))) }
 \end{align} 
 the $\odeg(w;\bfi((\npm)))$ refers to the degree in~\ref{setting shuffle} and $s_{(\npm)^0}$ stands for the notation in~\ref{symbol partition}. For other notations, we refer readers to Theorem~\ref{cohomology of rigid quiver grassmannians},
\end{Theorem}

\end{document}